\documentclass[a4paper]{article}
\usepackage{fullpage}
\usepackage{xfrac}
\usepackage{textcomp}
\usepackage{amsfonts}
\usepackage{amsmath}
\usepackage{mathtools}
\usepackage{ntheorem}
\usepackage{amssymb}

\usepackage[hidelinks]{hyperref}
\usepackage{listings}
\usepackage{graphicx}
\let\inf\relax \DeclareMathOperator*\inf{\vphantom{p}inf}
\let\liminf\relax \DeclareMathOperator*\liminf{\vphantom{p}liminf}
\let\lim\relax \DeclareMathOperator*\lim{\vphantom{p}lim}
\let\max\relax \DeclareMathOperator*\max{\vphantom{p}max}
\newtheorem{theorem}{Theorem}

 \numberwithin{dummy}{section}
\newtheorem{proposition}{Proposition}
\newtheorem{lemma}{Lemma}

\newenvironment{proof}{\noindent\\ \noindent{\sc    Proof}}{{\samepage\par\nopagebreak\hbox to\hsize{\hfill$\Box$}}}

\newcommand{\be}{\begin{equation}}
\newcommand{\ee}{\end{equation}}
\addcontentsline{toc}{section}{Abstract}
\usepackage[nottoc,numbib]{tocbibind}

\numberwithin{equation}{section}
\newcommand{\rP}{\mathrm{P\space }} 
\newcommand{\rE}{\mathrm{E\space }} 
\newcommand{\rV}{\mathrm{Var\hspace{0.2mm}}} 
\newcommand{\rC}{\mathrm{Cov\hspace{0.2mm}}}

\begin{document}

\title{Critical branching as a pure death process  coming down from infinity
}

\author{ Serik Sagitov \\
Chalmers University of Technology and University of Gothenburg}

\date{} 
\maketitle

\begin{abstract}
We consider the critical Galton-Watson process with overlapping generations stemming from a single founder.
Assuming that both the variance of the offspring number and the average generation length are finite, we establish the convergence of the finite-dimensional distributions, conditioned on non-extinction at a remote time of observation. The limiting process is identified as a pure death process coming down from infinity.

This result brings a new perspective on  Vatutin's dichotomy claiming that in the critical regime of age-dependent reproduction, an extant population either contains a large number of short-living individuals or consists of few long-living individuals.

\end{abstract}

\section{Introduction
}\label{sec:int}

Consider a self-replicating system evolving in the discrete time setting according to the next rules:
 \begin{description}
\item[\ \ \,-]  the system is founded by a single individual, the founder  born at time 0,
\item[\ \ \,-]  the founder dies at a random age $L$ and gives a random number $N$ of births at random ages $\tau_j$ satisfying
\begin{equation*}
 1\le\tau_1\le \ldots\le \tau_N\le L, 
\end{equation*}
\item[\ \ \,-]  each new individual lives independently from others according to the same life law as the founder.
\end{description}
An individual which was born at time $t_1$ and dies at time $t_2$ is considered to be  alive during the time interval $[t_1,t_2-1]$. Letting $Z(t)$ stand for the number of individuals alive at time $t$, we study the random dynamics of the sequence
$$Z(0)=1, Z(1), Z(2),\ldots,$$ 
which is a natural extension of the well-known Galton-Watson process, or \textit{GW-process} for short, 
see \cite{WG}. 
The process $Z(\cdot)$ is the discrete time version of what is usually called the Crump-Mode-Jagers process or the general branching process, see \cite{J}. 
 To emphasise the discrete time setting, we call it a GW-process with overlapping generations, or \textit{GWO-process} for short.
 
Put $b:=\frac{1}{2}\rV(N)$. This paper deals with the GWO-processes satisfying
 \begin{equation}
 \label{ir}
\rE(N)=1,\quad 0<b<\infty.
\end{equation}
Condition $\rE(N)=1$ says that the reproduction regime is critical, implying $\rE(Z(t))\equiv1$ and making extinction inevitable, provided $b>0$.
According to  \cite[Ch I.9]{AN}, given \eqref{ir},
 the survival probability 
 $$Q(t):=\rP(Z(t)>0)$$
  of a GW-process satisfies the asymptotic formula 
$tQ(t)\to b^{-1}$ as $t\to\infty$ (this was first proven in \cite{K} under a third moment assumption). A direct extension of this classical result for the GWO-processes, $$tQ(ta)\to b^{-1},\quad t\to\infty,\quad a:=\rE(\tau_1+\ldots+\tau_N),$$
was obtained in \cite{D,H} under conditions \eqref{ir}, $a<\infty$, 
\begin{equation}
  \label{d0}
t^2\rP(L>t)\to 0,\quad t\to\infty,
\end{equation}
plus an additional extra condition. (Notice that  by our definition, $a\ge1$, and $a=1$ if and only if $L\equiv1$, that is when the GWO-process in question is a GW-process). Treating $a$  as the \textit{mean generation length}, see  \cite{J,21}, we may conclude that the asymptotic behaviour of the critical GWO-process with \textit{short-living individuals}, see condition \eqref{d0}, is similar to that of the critical GW-process, provided time is counted generation-wise. 

New asymptotical patterns for the critical GWO processes are found under the assumption
 \begin{equation}
  \label{d}
t^2\rP(L>t)\to d,\quad 0\le d< \infty,\quad t\to\infty,
\end{equation}
which compared to \eqref{d0}, allows the existence of \textit{long-living individuals} given $d>0$.  Condition \eqref{d} was first introduced in the pioneering paper \cite{V79} dealing with the \textit{Bellman-Harris processes}.  In the current discrete time setting, the Bellman-Harris process is a GWO-process subject to two restrictions:  
 \begin{description}
\item[\ \ \,-]  $\rP(\tau_1=\ldots=\tau_N= L)=1$, so that all births occur at the moment of individual's death,
\item[\ \ \,-]   the random variables $L$ and $N$ are independent. 
\end{description}
For the Bellman-Harris process, conditions \eqref{ir} and \eqref{d} imply $a=\rE(L)$,  $a<\infty$, and according to  \cite[Theorem 3]{V79}, we get
 \begin{equation}
  \label{ad}
tQ(t)\to h,\quad t\to\infty,\qquad  h:=\frac{a+\sqrt{a^2+4bd}}{2b}.
\end{equation}
As was shown in  \cite[Corollary B]{T}, see also \cite[Lemma 3.2]{95} for an adaptation to the discrete time setting, relation  \eqref{ad} holds even for the GWO-processes satisfying conditions \eqref{ir}, \eqref{d}, and $a<\infty$.

The main result of this paper, Theorem 1 of Section \ref{main}, considers a critical GWO-process under the above mentioned neat set of assumptions \eqref{ir}, \eqref{d}, $a<\infty$, and establishes the convergence of the finite-dimensional distributions conditioned on survival at a remote time of observation. A remarkable feature of this result is that its limit process is fully described by a single parameter $c:=4bda^{-2}$, regardless of complicated mutual dependencies  between the random variables $\tau_j, N,L$.  

Our proof of Theorem 1, requiring an intricate asymptotic analysis of multi-dimensional probability generating functions, for the sake of readability, is split into two sections. Section \ref{out} presents a new proof of \eqref{ad} inspired by the proof of \cite{V79}. The crucial aspect of this approach, compared to the proof of \cite[Lemma 3.2]{95},
 is that certain essential steps do not rely on the monotonicity of the function $Q(t)$. In Section \ref{Lp1},  the technique of Section \ref{out} is further developed to finish the proof of Theorem 1. 

We conclude this section by mentioning the illuminating family of GWO-processes called the \textit{Sevastyanov  processes} \cite{Sev}.  The Sevastyanov process is a generalised version of the Bellman-Harris process, with possibly dependent $L$ and $N$. In the critical case, the mean generation length of the Sevastyanov process, $a=\rE(L N)$, can be represented as
$$a=\rC(L,N)+\rE(L).$$
Thus, if $L$ and $N$ are positively correlated, the average generation length $a$ exceeds the average life length $\rE(L)$.

Turning to a specific example of the Sevastyanov process, take
\[\rP(L= t)= p_1 t^{-3}(\ln\ln t)^{-1},  \quad \rP(N=0|L= t)=1-p_2,\quad \rP(N=n_t|L= t)=p_2,  \ t\ge2,\]
where $n_t:=\lfloor t(\ln t)^{-1}\rfloor$ and $(p_1,p_2)$ are such that
\[\sum_{t=2}^\infty  \rP(L= t)=p_1 \sum_{t=2}^\infty  t^{-3}(\ln\ln t)^{-1}=1,\quad \rE(N)=p_1p_2\sum_{t=2}^\infty  n_t t^{-3}(\ln\ln t)^{-1}=1.\]
In this case, for some positive constant $c_1$,
\[\rE(N^2)= p_1p_2\sum_{t=1}^\infty  n_t^2 t^{-3}(\ln\ln t)^{-1}< c_1\int_2^\infty \frac{d (\ln t)}{(\ln t)^2\ln\ln t}<\infty,\]
implying that condition \eqref{ir} is satisfied. Clearly, condition \eqref{d} holds with $d=0$.
At the same time,
\[a=\rE(NL)= p_1p_2\sum_{t=1}^\infty  n_t t^{-2}(\ln\ln t)^{-1}>   c_2\int_2^\infty \frac{d (\ln t)}{(\ln t)(\ln\ln t)}=\infty,\]
where $c_2$ is a positive constant. This example demonstrates that for the GWO-process, unlike the Bellman-Harris process, 
conditions \eqref{ir} and \eqref{d} do not automatically imply the condition $a<\infty$.

\section{The main result}\label{main}

\begin{theorem}\label{thL} 
For a GWO-process satisfying \eqref{ir}, \eqref{d} and $a<\infty$, there holds a weak convergence of the finite dimensional distributions
\begin{align*}
(Z(ty),0<y<\infty|Z(t)>0)\stackrel{\rm fdd\,}{\longrightarrow} (\eta(y),0<y<\infty),\quad t\to\infty.
\end{align*}
The limiting process is a continuous time pure death process $(\eta(y),0\le y<\infty)$, 
whose evolution law is determined by a single compound parameter $c=4bda^{-2}$, as specified next. 
\end{theorem}

The finite dimensional distributions of the limiting process $\eta(\cdot)$ are given below in  terms of the $k$-dimensional probability generating functions
$\rE(z_1^{\eta(y_1)}\cdots z_k^{\eta(y_k)})$, $k\ge1$, assuming
\begin{equation}\label{mansur}
 0=y_0< y_1< \ldots<  y_{j}<1\le y_{j+1}< \ldots<  y_k<y_{k+1}=\infty,\quad 0\le j\le k,\quad 0\le z_1,\ldots,z_k<1.
\end{equation}
Here the index $j$ highlights the pivotal value 1 corresponding to the time of observation $t$ of the underlying GWO-process. 

As will be shown in Section \ref{Send}, if $j=0$, then
\begin{align*}
\rE(z_1^{\eta(y_1)}\cdots z_k^{\eta(y_k)})=1-\frac{1+\sqrt{1+\sum\nolimits_{i=1}^{k}z_1\cdots z_{i-1}(1-z_{i})\Gamma_i}}{(1+\sqrt{1+c})y_1},\quad \Gamma_i:=c({y_1}/{y_i} )^2,
\end{align*}
and
if $j\ge1$,
\begin{align*}
\rE(z_1^{\eta(y_1)}\cdots z_k^{\eta(y_k)})=\frac{\sqrt{1+\sum_{i=1}^{j}z_1\cdots z_{i-1}(1-z_{i})\Gamma_i+cz_1\cdots z_{j}y_1^{2} }-\sqrt{1+\sum\nolimits_{i=1}^{k}z_1\cdots z_{i-1}(1-z_{i})\Gamma_i}}{(1+\sqrt{1+c})y_1}.
\end{align*} 
In particular, for $k=1$, we have
\begin{align*}
\rE(z^{\eta(y)})&= \frac{\sqrt{1+c(1-z)+czy^{2}}-\sqrt{1+c(1-z)}}{(1+\sqrt{1+c})y},\quad 0< y<1,\\
\rE(z^{\eta(y)})&= 1-\frac{1+\sqrt{1+c(1-z)}}{(1+\sqrt{1+c})y},\quad y\ge1. 
\end{align*}
It follows that $\rP(\eta(y)\ge0)=1$ for $y>0$, and moreover, putting here first $z=1$ and then $z=0$, brings
\begin{align*}
\rP(\eta(y)<\infty)&=\frac{\sqrt{1+cy^2}-1}{(1+\sqrt{1+c})y}\cdot1_{\{0< y<1\}}+\Big(1-\frac{2}{(1+\sqrt{1+c})y}\Big)\cdot1_{\{y\ge 1\}},\\
\rP(\eta(y)=0)&=\frac{y-1}{y}\cdot1_{\{y\ge 1\}},
\end{align*}
implying that $\rP(\eta(y)=\infty)>0$ for all $y>0$, and in fact, letting $y\to0$, we may set
$\rP(\eta(0)=\infty)=1.$

To demonstrate that  the process $\eta(\cdot)$ is indeed a pure death process, consider the function
\[\rE(z_1^{\eta(y_1)-\eta(y_2)}\cdots z_{k-1}^{\eta(y_{k-1})-\eta(y_{k})}z_k^{\eta(y_k)})\]
determined by
\begin{align*}
\rE(z_1^{\eta(y_1)-\eta(y_2)}\cdots z_{k-1}^{\eta(y_{k-1})-\eta(y_{k})}z_k^{\eta(y_k)})
&=\rE(z_1^{\eta(y_1)}(z_2/z_1)^{\eta(y_2)}\cdots (z_k/z_{k-1})^{\eta(y_k)}).
\end{align*}
This function is given by two expressions
\begin{align*}
\frac{(1+\sqrt{1+c})y_1-1-\sqrt{1+\sum\nolimits_{i=1}^{k} (1-z_{i})\gamma_i}}{(1+\sqrt{1+c})y_1}, \quad &\text{for }j=0,\\
\frac{\sqrt{1+\sum\nolimits_{i=1}^{j-1}(1-z_{i})\gamma_i+(1-z_{j})\Gamma_j+cz_j y_1^2}-\sqrt{1+\sum\nolimits_{i=1}^{k} (1-z_{i})\gamma_i}}{(1+\sqrt{1+c})y_1}, \quad &\text{for }j\ge1,
\end{align*}
where $\gamma_i:=\Gamma_i-\Gamma_{i+1}$ and $\Gamma_{k+1}=0$.
Setting $k=2$, $z_1=z$, and $z_2=1$, we deduce that the function
\begin{equation}\label{lava}
 \rE(z^{\eta(y_1)-\eta(y_2)};\eta(y_1)<\infty),\quad 0<y_1<y_2,\quad 0\le z\le1,
\end{equation}
is given by one of the following three expressions depending on whether $j=2$, $j=1$, or $j=0$,
\begin{align*}
\frac{\sqrt{1+c y_1^2+c(1-z)(1-(y_1/y_2)^2)}
-\sqrt{1+c (1-z)(1-(y_1/y_2)^2)}}{(1+\sqrt{1+c})y_1},\quad &y_2<1,  \\
\frac{\sqrt{1+c y_1^2+c(1-z) (1-y_1^2)}
-\sqrt{1+c(1-z)(1-(y_1/y_2)^2)}}{(1+\sqrt{1+c})y_1},\quad &y_1<1\le y_2, \\
1- \frac{1+\sqrt{1+c(1-z)(1-(y_1/y_2)^2)}}{(1+\sqrt{1+c})y_1},\quad &1\le y_1.  
\end{align*}
Since generating function \eqref{lava} is finite at $z=0$, we conclude that 
$$\rP(\eta(y_1)< \eta(y_2); \eta(y_1)< \infty)=0,\quad 0<y_1<y_2.$$
 This implies 
 $$\rP(\eta(y_2)\le \eta(y_1))=1,\quad 0<y_1<y_2,$$
meaning that unless the process $\eta(\cdot)$ is sitting at the infinity state, it evolves by negative integer-valued jumps until it gets absorbed at zero. 

Consider now the conditional probability generating function
\begin{equation}\label{ava}
\rE(z^{\eta(y_1)-\eta(y_2)}| \eta(y_1)<\infty),\quad 0<y_1<y_2,\quad 0\le z\le1.
\end{equation}
In accordance with the above given three expressions for \eqref{lava}, generating function \eqref{ava} is specified by the following three expressions
\begin{align*}
\frac{\sqrt{1+c y_1^2+c(1-z)(1-(y_1/y_2)^2)}
-\sqrt{1+c (1-z)(1-(y_1/y_2)^2)}}{\sqrt{1+c y_1^2}-1},\quad &y_2<1,  \\
\frac{\sqrt{1+c y_1^2+c(1-z) (1-y_1^2)}
-\sqrt{1+c(1-z)(1-(y_1/y_2)^2)}}{\sqrt{1+c y_1^2}-1},\quad &y_1<1\le y_2, \\
1- \frac{\sqrt{1+c(1-z)(1-(y_1/y_2)^2)}-1}{(1+\sqrt{1+c})y_1-2},\quad &1\le y_1.  
\end{align*}
In particular, setting here $z=0$, we obtain
\[\rP(\eta(y_1)-\eta(y_2)=0| \eta(y_1)<\infty)= \left\{
\begin{array}{llr}
\frac{\sqrt{1+c(1+y_1^2-(y_1/y_2)^2)}-\sqrt{1+c(1-(y_1/y_2)^2)}}{\sqrt{1+c y_1^2}-1} &  \text{for}  &  0<y_1< y_2<1, \\
\frac{\sqrt{1+c}-\sqrt{1+c(1-(y_1/y_2)^2)}}{\sqrt{1+c y_1^2}-1} &  \text{for}  &  0<y_1<1\le y_2, \\
1- \frac{\sqrt{1+c(1-(y_1/y_2)^2)}-1}{(1+\sqrt{1+c})y_1-2} & \text{for}   &   1\le y_1<y_2.
\end{array}
\right.
\]
Notice that given $0<y_1\le1$, 
\[\rP(\eta(y_1)-\eta(y_2)=0| \eta(y_1)<\infty)\to 0,\quad y_2\to\infty,\]
which is expected because of $\eta(y_1)\ge\eta(1)\ge1$ and $\eta(y_2)\to0$ as $y_2\to\infty$.

\begin{figure}[h]
\begin{center}
\includegraphics[width=6cm]{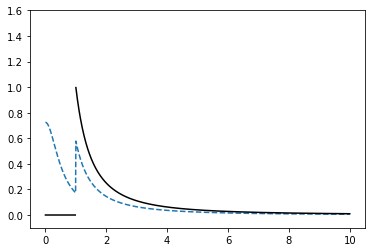}\qquad \includegraphics[width=6cm]{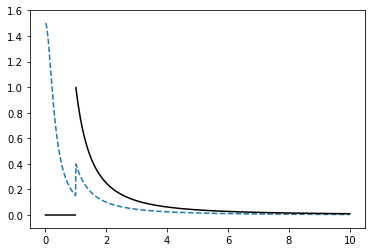}
\end{center}
\caption{The dashed line is the probability density function of $T$, the solid line is the probability density function of $T_0$. The left panel illustrates the case $c=5$, and the right panel  illustrates the case $c=15$.}
\label{trump}
 \end{figure}

The random times
\[T=\sup\{u: \eta(u)=\infty\},\quad T_0=\inf\{u:\eta(u)=0\},\]
are major characteristics of a trajectory of the limit pure death process. 
Since
\begin{align*}
\rP(T\le y)=\rE(z^{\eta(y)})\Big\vert_{z=1},\qquad  \rP(T_0\le y)=\rE(z^{\eta(y)})\Big\vert_{z=0},
\end{align*}
in accordance with the above mentioned formulas for $\rE(z^{\eta(y)})$, we get the following marginal distributions
\begin{align*}
 \rP(T\le y)&=\frac{\sqrt{1+cy^2}-1}{(1+\sqrt{1+c})y}\cdot1_{\{0\le y<1\}}+\Big(1-\frac{2}{(1+\sqrt{1+c})y}\Big)\cdot1_{\{y\ge 1\}},\\
 \rP(T_0\le y)&=\frac{y-1}{y}\cdot1_{\{y\ge 1\}}.
\end{align*}
The distribution of $T_0$ is free from the parameter $c$ and has the Pareto probability density function 
\[f_0(y)=y^{-2}1_{\{y>1\}}.\]
In the special case \eqref{d0}, that is when  \eqref{d} holds with $d=0$, we have $c=0$ and 
$\rP(T=T_0)=1$.
If $d>0$, then $T\le T_0$, and the distribution of $T$  has the following probability density function 
\[
f(y)=\left\{
\begin{array}{llr}
\frac{1}{(1+\sqrt{1+c})y^2} (1-\frac{1}{\sqrt{1+cy^2}})&  \text{for}  &  0\le y<1, \\
\frac{2}{(1+\sqrt{1+c})y^2} & \text{for}   &   y\ge1,
\end{array}
\right.
\]
having a positive jump at $y=1$ of size $f(1)-f(1-)=(1+c)^{-1/2}$. Observe that
$\frac{f(1-)}{f(1)}\to\frac{1}{2}$ as $c\to\infty$.

Intuitively, the limiting pure death process counts the long-living individuals in the GWO-process, that is those individuals whose life length is of order $t$. These long-living individuals may have descendants, however none of them would live long enough to be detected by the finite dimensional distributions at the relevant time scale, see Lemma 2 below.
Theorem \ref{thL} suggests a new perspective on Vatutin's dichotomy, see \cite{V79}, claiming that the long term survival of a critical age-dependent branching process is due to either a large number of short-living individuals or a small number of long-living individuals. 
In terms of the random times $T\le T_0$, Vatutin's dichotomy discriminates between two possibilities: if $T>1$, then $\eta(1)=\infty$, meaning that the GWO-process has survived due to a large number of individuals, while if $T\le 1<T_0$,  then $1\le \eta(1)<\infty$ meaning  that the GWO-process has survived due to a small number of individuals.

\section{Proof of \ $\boldsymbol{tQ(t)\to h}$}\label{out}

This section deals with the survival probability of the critical GWO-process 
$$Q(t)=1-P(t),\quad P(t):=\rP(Z(t)=0).$$
By its definition, the GWO-process can be represented as the sum 
\begin{equation}\label{CD}
 Z(t)=1_{\{L>t\}}+\sum\nolimits_{j=1}^{N}  Z_j(t-\tau_j),\quad t=0,1,\ldots,
\end{equation}
involving $N$ independent daughter processes $Z_j(\cdot)$  generated by the founder individual at the birth times $\tau_j$, $j=1,\ldots,N$ (here it is assumed that  $Z_j(t)=0$ for all negative $t$).
The branching property \eqref{CD} implies the relation
\[ 1_{\{Z(t)=0\}}=1_{\{L\le  t\}}\prod\nolimits_{j=1}^{N} 1_{\{Z_j (t-\tau_j)=0\}},\]
saying that the GWO-process goes extinct by the time $t$ if, on one hand, the founder is dead at time $t$ and, on the other hand, all daughter processes are extinct by the time $t$.
After taking expectations of both sides, we can write
\begin{equation}\label{ejp}
P(t)=\rE\Big(\prod\nolimits_{j=1}^{N}P(t-\tau_j);L\le t\Big).
\end{equation}
As shown next, this non-linear equation for $P(\cdot)$  entails the asymptotic formula \eqref{ad} under conditions \eqref{ir}, \eqref{d}, and $a<\infty$.

\subsection{Outline of the proof of \eqref{ad}}\label{ou}

We start by stating four lemmas and two propositions. Let
 \begin{align} 
\Phi(z)&:=\rE((1-z)^ N-1+Nz), \label{AL}\\
W(t)&:=(1-ht^{-1})^{N}+Nht^{-1}-\sum\nolimits_{j=1}^{N}Q(t-\tau_j)-\prod\nolimits_{j=1}^{N} P(t-\tau_j), \label{Wt}\\
D(u,t)&:=\rE\Big(1-\prod\nolimits_{j=1}^{N}P(t-\tau_j);\,u<L\le t\Big)+\rE\Big((1-ht^{-1})^{N} -1+Nht^{-1};L> u\Big), \label{Dut}\\
\rE_u(X)&:=\rE(X;L\le u ),\label{Et}
\end{align}
where $0\le z\le 1$, $u>0$, $t\ge h$, and $X$ is an arbitrary random variable.

\begin{lemma}\label{fQd}
Given  \eqref{AL}, \eqref{Wt}, \eqref{Dut}, and \eqref{Et}, assume that $0< u\le t$ and $t\ge h$. Then
 \begin{align*} 
\Phi(ht^{-1})= \rP(L> t)+\rE_u\Big(\sum\nolimits_{j=1}^{N}Q(t-\tau_j)\Big)-Q(t)+\rE_u(W(t))+D(u,t). 
\end{align*}
\end{lemma}

\begin{lemma}\label{L3}
If \eqref{ir} and \eqref{d} hold, then $\rE(N;L>ty)=o(t^{-1})$ as $t\to\infty$ for any fixed $y>0$.
\end{lemma}

\begin{lemma}\label{L2}
 If \eqref{ir}, \eqref{d}, and $a<\infty$ hold, then for any fixed $0<y<1$, 
\begin{align*}
\rE_{ty}\Big(\sum\nolimits_{j=1}^{N}\Big(\frac{1}{t-\tau_j}-\frac{1}{t}\Big)\Big)\sim at^{-2},\quad t\to\infty.
\end{align*}
\end{lemma}

\begin{lemma}\label{L4}
Let $k\ge1$. If $0\le f_j,g_j\le 1$  for $ j=1,\ldots,k$, then
\[ \prod\nolimits_{j=1}^k(1-g_j)-\prod\nolimits_{j=1}^k(1-f_j)=\sum\nolimits_{j=1}^k (f_j-g_j)r_j, \]
where $0\le r_j\le1$ and
\begin{align*}
1-r_j=\sum\nolimits_{i=1}^{j-1}g_i+\sum\nolimits_{i=j+1}^{k}f_i-R_j,
\end{align*}
for some $R_j\ge0$.
If moreover, $f_j\le  q$ and $g_j\le  q$ for some $q>0$, then
$$1- r_j\le(k-1)q,\qquad R_j\le kq,\qquad R_j\le k^2q^2.$$
\end{lemma}

\begin{proposition}\label{Lx}
 If \eqref{ir}, \eqref{d}, and $a<\infty$ hold, then $\limsup_{t\to\infty} tQ(t)<\infty$.
 \end{proposition}

\begin{proposition}\label{Ly}
 If \eqref{ir}, \eqref{d}, and $a<\infty$ hold, then $\liminf_{t\to\infty} tQ(t)>0$.
 \end{proposition}

\vspace{0.25cm}
According to these two propositions, there exists a triplet of positive numbers $(q_1,q_2,t_0)$ such that 
\begin{equation}\label{ca}
q_1\le tQ(t)\le q_2,\quad t\ge t_0,\quad 0<q_1<h<q_2<\infty.
\end{equation}
The claim $tQ(t)\to h$ is derived using \eqref{ca} by accurately removing asymptotically negligible terms from the relation for $Q(\cdot)$ stated in Lemma 1, after setting $u=ty$ with a fixed $0<y<1$, and then choosing a sufficiently small $y$. In particular, as an intermediate step, we will show that
\begin{align}
Q(t)= \rE_{ty}\Big(\sum\nolimits_{j=1}^{N}Q(t-\tau_j)\Big)+\rE_{ty}(W(t))-aht^{-2}+o(t^{-2}),\quad t\to\infty. \label{rys}
\end{align}
Then, restating our goal as $\phi(t)\to 0$ in terms of the function $\phi(t)$, defined by
\begin{equation}\label{cal}
Q(t)=\frac{h +\phi(t)}{t},\quad t\ge1,
\end{equation}
we rewrite \eqref{rys} as
 \begin{align} 
\frac{h +\phi(t)}{t}&= \rE_{ty}\Big(\sum\nolimits_{j=1}^{N}\frac{h +\phi(t-\tau_j)}{t-\tau_j}\Big)+\rE_{ty}(W(t))-aht^{-2}+o(t^{-2}),\quad t\to\infty. \label{eye}
\end{align}
It turns out that the three terms involving $h$, outside $W(t)$, effectively cancel each other, yielding
\begin{align}
\frac{\phi(t)}{t}&= \rE_{ty}\Big(\sum\nolimits_{j=1}^{N}\frac{\phi(t-\tau_j)}{t-\tau_j}+W(t)\Big)+o(t^{-2}),\quad t\to\infty.\label{luh}
\end{align}

Treating $W(t)$ in terms of Lemma \ref{L4}, brings
\begin{align}
 \phi(t)&= \rE_{ty}\Big(\sum\nolimits_{j=1}^{N}\phi(t-\tau_j) r_j(t)\frac{t}{t-\tau_j}\Big)+o(t^{-1}), \label{afo}
\end{align}
where $r_j(t)$ is a counterpart of $r_j$ in Lemma \ref{L4}.
To derive from here the desired convergence $\phi(t)\to0$, we will adapt a clever trick from Chapter 9.1 of \cite{Seva}, which was further developed in  \cite{V79} for the Bellman-Harris process, with possibly  infinite $\rV(N)$.  Define a non-negative function $m(t)$ by
\begin{align}
m(t):=|\phi(t)|\, \ln t,\quad t\ge 2.  \label{mt}
\end{align}
  Multiplying \eqref{afo} by $\ln t$ and using the triangle inequality, we obtain
\begin{align}
 m(t)\le \rE_{ty}\Big(\sum\nolimits_{j=1}^{N} m(t-\tau_j)r_j(t) \frac{t\ln t}{(t-\tau_j)\ln(t-\tau_j)}\Big)+v(t),\label{elin}
\end{align}
where $v(t)\ge 0$ and $v(t)=o(t^{-1}\ln t)$ as $t\to\infty$. It will be shown that this leads to $m(t)=o(\ln t)$, thereby concluding the proof of \eqref{ad}.

\subsection{Proof of lemmas and  propositions}\label{lemmas}

\begin{proof} {\sc of Lemma \ref{fQd}}. 
For $0<u\le t$, relations \eqref{ejp} and  \eqref{Et} give
\begin{align}\label{Qln}
P(t)=\rE_u\Big(\prod\nolimits_{j=1}^{N}P(t-\tau_j) \Big)+\rE\Big(\prod\nolimits_{j=1}^{N}P(t-\tau_j);u<L\le  t\Big).
\end{align}
On the other hand, for $t\ge h$,
\begin{align*} 
\Phi(ht^{-1})
&\stackrel{\eqref{AL}}{=}\rE_u\Big((1-ht^{-1})^{N}-1+Nht^{-1}\Big)+\rE\Big((1-ht^{-1})^{N}-1 +Nht^{-1};L> u\Big).
\end{align*}
Adding the latter relation  to
\begin{align*} 
1
&=\rP(L\le u)+\rP(L> t)+\rP(u<L\le t),
\end{align*}
and subtracting \eqref{Qln} from the sum, we get
\begin{align*}
\Phi(ht^{-1})+Q(t)=\rE_u\Big((1-ht^{-1})^{N} +Nht^{-1}-\prod\nolimits_{j=1}^{N}P(t-\tau_j)\Big)+\rP(L> t)+D(u,t),
\end{align*}
with $D(u,t)$ defined by \eqref{Dut}.
After a rearrangement, we obtain the statement of the lemma.
\end{proof}

\begin{proof}   {\sc of Lemma \ref{L3}}.
For any fixed $\epsilon>0$,
 \begin{align*}
\rE(N;L>t)=\rE(N;N\le t\epsilon,L>t)+\rE(N;1<N(t\epsilon)^{-1},L>t)\le t\epsilon\rP(L>t)+(t\epsilon)^{-1}\rE(N^2;L>t).
\end{align*}
Thus, by \eqref{ir} and \eqref{d},
 \begin{align*}
\limsup_{t\to\infty} (t\rE(N;L>t))\le d\epsilon,
\end{align*}
and the assertion follows as $\epsilon\to0$.
\end{proof}

\begin{proof} {\sc of Lemma \ref{L2}}.
For $t=1,2,\ldots$ and $y>0$, put
\begin{align*}
B_t(y)&:= t^2\,\rE_{ty}\Big(\sum\nolimits_{j=1}^{N}\Big(\frac{1}{t-\tau_j}-\frac{1}{t}\Big)\Big)-a.
 \end{align*}
For any  $0<u<ty$, using
 \[a=\rE_u(\tau_1+\ldots+\tau_N)+A_u,\quad A_u:=\rE(\tau_1+\ldots+\tau_N;L> u),\]
we get
\begin{align*}
B_t(y)&= \rE_u\Big(\sum\nolimits_{j=1}^{N} \frac{t}{t-\tau_j}\tau_j\Big)+\rE\Big(\sum\nolimits_{j=1}^{N} \frac{t}{t-\tau_j}\tau_j\,;u<L\le ty\Big)-\rE_u(\tau_1+\ldots+\tau_N)-A_u
\\
 &=\rE\Big(\sum\nolimits_{j=1}^{N}\frac{\tau_j}{1-\tau_j/t};u<L\le ty\Big)+\rE_u\Big(\sum\nolimits_{j=1}^{N}\frac{\tau_j^2}{t-\tau_j}\Big)-A_u.
\end{align*}
For the first term on the right hand side, we have $\tau_j\le L\le ty$, so that
\begin{align*}
\rE\Big(\sum\nolimits_{j=1}^{N}\frac{\tau_j}{1-\tau_j/t};u<L\le ty\Big)\le(1-y)^{-1}A_u.
\end{align*}
For the second term, $\tau_j\le L\le u$ and therefore
\begin{align*}
\rE_u\Big(\sum\nolimits_{j=1}^{N}\frac{\tau_j^2}{t-\tau_j}\Big)\le\frac{u^2}{t-u}\rE_u(N)\le\frac{u^2}{t-u}.
\end{align*}

This yields
\[-A_u\le B_t(y)\le (1-y)^{-1}A_u+\frac{u^2}{t-u},\quad 0<u<ty<t,\]
implying
\[-A_u\le \liminf_{t\to\infty} B_t(y)\le\limsup_{t\to\infty} B_t(y)\le (1-y)^{-1}A_u.\]
Since $A_u\to0$ as $u\to\infty$, we conclude that  $B_t\to0$ as $t\to\infty$.
\end{proof}

\begin{proof} {\sc of Lemma \ref{L4}}.
Let 
\begin{equation*}
r_j:=(1-g_1)\ldots (1-g_{j-1})(1-f_{j+1})\ldots (1-f_k),\quad 1\le j\le k.
\end{equation*}
Then $0\le r_j\le1$ and the first stated equality is obtained by telescopic summation of
\begin{align*}
 (1-g_1)\prod\nolimits_{j=2}^{k}(1-f_j)-\prod\nolimits_{j=1}^k(1-f_j)&=(f_1-g_1)r_1,\\
 (1-g_1)(1-g_2)\prod\nolimits_{j=3}^{k}(1-f_j)- (1-g_1)\prod\nolimits_{j=2}^{k}(1-f_j)&=(f_2-g_2)r_2,\ldots,\\
\prod\nolimits_{j=1}^{k}(1-g_j)-\prod\nolimits_{j=1}^{k-1}(1-g_j)(1-f_k)&=(f_k-g_k)r_k.
\end{align*}
The second stated equality is obtained with
\begin{align*}
R_j&:=\sum_{i=j+1}^{k}f_i(1-(1-f_{j+1})\ldots (1-f_{i-1}))+\sum_{i=1}^{j-1}g_i(1-(1-g_1)\ldots (1-g_{i-1})(1-f_{j+1})\ldots (1-f_k)),
\end{align*}
by performing telescopic summation of
\begin{align*}
 1-(1-f_{j+1})&=f_{j+1},\\
(1-f_{j+1})-(1-f_{j+1})(1-f_{j+2})&=f_{j+2}(1-f_{j+1}),\ldots,\\
 \prod\nolimits_{i=j+1}^{k-1}(1-f_i)- \prod\nolimits_{i=j+1}^{k}(1-f_i)&=f_k\prod\nolimits_{i=j+1}^{k-1}(1-f_i),\\
  \prod\nolimits_{i=j+1}^{k}(1-f_i)-(1-g_1)\prod\nolimits_{i=j+1}^{k}(1-f_i)&=g_1\prod\nolimits_{i=j+1}^{k}(1-f_i),\ldots,\\
  \prod\nolimits_{i=1}^{j-2}(1-g_i)\prod\nolimits_{i=j+1}^{k}(1-f_i)- \prod\nolimits_{i=1}^{j-1}(1-g_i)\prod\nolimits_{i=j+1}^{k}(1-f_i)&=g_{j-1} \prod\nolimits_{i=1}^{j-2}(1-g_i)\prod\nolimits_{i=j+1}^{k}(1-f_i).
\end{align*}

By the above definition of $R_j$, we have $R_j\ge0$. Furthermore, given $f_j\le  q$ and $g_j\le  q$, we get
\[R_j\le \sum\nolimits_{i=1}^{j-1}g_i+\sum\nolimits_{i=j+1}^{k}f_i\le (k-1)q.
\]
It remains to observe that
\begin{align*}
1-r_j\le 1-(1-q)^{k-1}\le (k-1)q,
\end{align*}
and from the definition of $R_j$,
\[R_j\le q\sum\nolimits_{i=1}^{k-j-1}(1-(1-q)^{i})+q\sum\nolimits_{i=1}^{j-1}(1-(1-q)^{k-j+i-1})\le q^2\sum\nolimits_{i=1}^{k-2}i\le k^2q^2.\]
\end{proof}

\begin{proof} {\sc of Proposition \ref{Lx}}.
By the definition of $\Phi(\cdot)$, we have
$$\Phi(Q(t))+P(t)=\rE_u\Big(P(t)^{N} \Big)+\rP(L> u)-\rE\Big(1-P(t)^ N;\,L>  u\Big),$$
for any $0<u<t$. This and \eqref{Qln} yield
 \begin{align}
\Phi(Q(t))&=\rE_u\Big(P(t)^{N}-\prod\nolimits_{j=1}^{N}P(t-\tau_j)\Big)+\rP(L>  u) \nonumber\\
&-\rE\Big(1-P(t)^ N;\,L>  u\Big)-\rE\Big(\prod\nolimits_{j=1}^{N}P(t-\tau_j);u<L\le  t\Big). \label{Nad}
\end{align}
An upper bound follows 
\begin{align*}
 \Phi(Q(t))&\le \rE_u\Big(P(t)^{N}-\prod\nolimits_{j=1}^{N}P(t-\tau_j)\Big)+\rP(L>  u),
\end{align*}
which together with Lemma \ref{L4} and monotonicity of $Q(\cdot)$ entail
\begin{align}\label{13}
 \Phi(Q(t))\le \rE_u\Big(\sum\nolimits_{j=1}^{N}(Q(t-\tau_j)-Q(t))\Big)+\rP(L>u).
\end{align}

Borrowing an idea from \cite{T}, suppose, on the contrary,  that 
$$t_n:=\min\{t: tQ(t)\ge n\}$$
 is finite for any natural $n$. It follows that 
$$Q(t_n)\ge \frac{n}{t_n},\qquad Q(t_n-u)<\frac{n}{t_n-u},\quad 1\le u\le t_n-1.$$ 
Putting $t=t_n$ into \eqref{13} and using monotonicity of $\Phi(\cdot)$, we find
\begin{eqnarray*}
 \Phi(nt_n^{-1})\le \Phi(Q(t_n))\le \rE_u\Big(\sum\nolimits_{j=1}^{N}\Big(\frac{n}{t_n-\tau_j}-\frac{n}{t_n}\Big)\Big)+\rP(L>  u).
\end{eqnarray*}
Setting here $u=t_n/2$ and applying Lemma \ref{L2} together with \eqref{d}, we arrive at the relation
$$\Phi(nt_n^{-1})=O(nt_n^{-2}),\quad n\to\infty.$$
Observe that under condition  \eqref{ir}, the L'Hospital rule gives
\begin{equation}\label{L1}
\Phi(z)\sim bz^2,\quad z\to0.
\end{equation}
The resulting contradiction, $n^{2}t_n^{-2}=O(nt_n^{-2})$ as $n\to\infty$, finishes the proof of the proposition. 
\end{proof}

\begin{proof} {\sc of Proposition \ref{Ly}}.
Relation \eqref{Nad} implies
\begin{align*}
 \Phi(Q(t))\ge \rE_u\Big(P(t)^{N}-\prod\nolimits_{j=1}^{N}P(t-\tau_j)\Big)-\rE\Big(1-P(t)^ N;\,L>  u\Big).
\end{align*}
By Lemma \ref{L4},
\begin{align*}
P(t)^{N}-\prod\nolimits_{j=1}^{N}P(t-\tau_j)= \sum_{j=1}^{N}(Q(t-\tau_j)-Q(t))r_j^*(t),
\end{align*}
where $0\le r_j^*(t)\le 1$ is a counterpart of term $r_j$ in Lemma \ref{L4}. Due to monotonicity of $P(\cdot)$, we have, again referring to Lemma \ref{L4},
$$1-r_j^*(t)\le (N-1)Q(t-L).$$ 
Thus, for $0<y<1$, 
\begin{align}\label{cont}
 \Phi(Q(t))&\ge \rE_{ty}\Big(\sum_{j=1}^{N}(Q(t-\tau_j)-Q(t))r_j^*(t) \Big)-\rE\Big(1-P(t)^ N;\,L>  ty\Big).
 \end{align}

The assertion $\liminf_{t\to\infty} tQ(t)>0$ is proven by contradiction. Assume that $\liminf_{t\to\infty} tQ(t)=0$, so that
$$t_n:=\min\{t: tQ(t)\le n^{-1}\}$$
is finite for any natural $n$. Plugging $t=t_n$ in \eqref{cont} and using
$$Q(t_n)\le \frac{1}{nt_n},\quad Q(t_n-u)-Q(t_n)\ge \frac{1}{n(t_n-u)}-\frac{1}{nt_n},\quad 1\le u\le t_n-1,$$
we get
$$\Phi\Big(\frac{1}{nt_n}\Big)\ge n^{-1}\rE_{t_ny}\Big(\sum\nolimits_{j=1}^{N}\Big(\frac{1}{t_n-\tau_j}-\frac{1}{t_n}\Big)r_j^*(t_n)\Big)-\frac{1}{nt_n}\rE(N;\,L> t_ny).$$
Given $L\le ty$, we have
\begin{align*}
1-r_j^*(t)\le NQ(t(1-y))\le N\frac{q_2}{t(1-y)},
 \end{align*}
 where the second inequality is based on the already proven part of \eqref{ca}. Therefore,
$$\rE_{t_ny}\Big(\sum\nolimits_{j=1}^{N}\Big(\frac{1}{t_n-\tau_j}-\frac{1}{t_n}\Big)(1-r_j^*(t_n))\Big)\le \frac{q_2y}{t_n^2(1-y)^2}\rE(N^2),$$
and we derive
\begin{align*}
 nt_n^2\Phi(\tfrac{1}{nt_n})&\ge  t_n^2\rE_{t_ny}\Big(\sum\nolimits_{j=1}^{N}\Big(\frac{1}{t_n-\tau_j}-\frac{1}{t_n}\Big)\Big)
 -\frac{\rE(N^2)q_2y}{(1-y)^2}-t_n\rE(N;\,L> t_ny).
\end{align*}
Sending $n\to\infty$ and applying \eqref{L1}, Lemma \ref{L3}, and Lemma \ref{L2},
we arrive at the inequality
$$0\ge a-yq_2\rE(N^2)(1-y)^{-2},\quad 0<y<1,$$
which is false for sufficiently small $y$.
\end{proof}

\subsection{Proof of \eqref{luh} and \eqref{afo}}\label{end}

Fix an arbitrary $0<y<1$.
Lemma 1 with $u=ty$, gives
\begin{align}
\Phi(h t^{-1})= \rP(L> t)+\rE_{ty}\Big(\sum\nolimits_{j=1}^{N}Q(t-\tau_j)\Big)-Q(t)+\rE_{ty}(W(t))+D(ty,t). \label{mam}
\end{align}
Let us show that
\begin{align}
D(ty,t)=o(t^{-2}),\quad t\to\infty. \label{cov}
\end{align}
Using Lemma \ref{L3} and \eqref{ca}, we find that for an arbitrarily small $\epsilon>0$, 
\[\rE\Big(1-\prod\nolimits_{j=1}^{N}P(t-\tau_j);\,ty<L\le t(1-\epsilon)\Big)=o(t^{-2}),\quad t\to\infty.\]
On the other hand,
   \begin{align*} 
\rE\Big(1-\prod\nolimits_{j=1}^{N}P(t-\tau_j);\,t(1-\epsilon)<L\le t\Big)\le \rP(t(1-\epsilon)<L\le t),
\end{align*}
so that in view of \eqref{d},
\[\rE\Big(1-\prod\nolimits_{j=1}^{N}P(t-\tau_j);\,ty<L\le t\Big)=o(t^{-2}),\quad t\to\infty.\]
This,  \eqref{Dut} and Lemma \ref{L3} entail \eqref{cov}.

Observe that
\begin{equation}\label{stop}
bh^2=ah+d.
\end{equation}
Combining \eqref{mam}, \eqref{cov},
and
$$\rP(L> t)-\Phi(h t^{-1})\stackrel{\eqref{d}\eqref{L1}}{=}dt^{-2}-bh^2t^{-2}+o(t^{-2})\stackrel{\eqref{stop}}{=}-aht^{-2}+o(t^{-2}),\quad t\to\infty,$$
we derive \eqref{rys}, which in turn gives \eqref{eye}. The latter implies \eqref{luh} since by Lemmas \ref{L3} and \ref{L4},
\[ \rE_{ty}\Big(\sum\nolimits_{j=1}^{N}\frac{h }{t-\tau_j}\Big)-\frac{h}{t}=\rE_{ty}\Big(\sum\nolimits_{j=1}^{N}\Big(\frac{h }{t-\tau_j}-\frac{h}{t}\Big)\Big)
-ht^{-1}\rE(N;L> ty)=aht^{-2}+o(t^{-2}).
\]
 
Turning to the proof of \eqref{afo}, observe that the random variable
$$
W(t)=(1-h t^{-1})^{N}-\prod\nolimits_{j=1}^{N}\Big(1-\frac{h +\phi(t-\tau_j)}{t-\tau_j}\Big)+\sum\nolimits_{j=1}^{N}\Big(\frac{h }{t}-\frac{h +\phi(t-\tau_j)}{t-\tau_j}\Big),
$$
can be represented in terms of Lemma \ref{L4} as
$$ 
W(t)=\prod\nolimits_{j=1}^{N}(1-f_j(t))-\prod\nolimits_{j=1}^{N}(1-g_j(t))+\sum\nolimits_{j=1}^{N}(f_j(t)-g_j(t))=\sum\nolimits_{j=1}^{N}(1-r_j(t))(f_j(t)-g_j(t)),
$$
by assigning
\begin{align}\label{sal}
f_j(t):=h t^{-1},\quad g_j(t):=\frac{h +\phi(t-\tau_j)}{t-\tau_j}.
\end{align}
Here $0\le r_j(t)\le 1$ and for sufficiently large $t$,
\begin{align}\label{stal}
1-r_j(t)\stackrel{ \eqref{ca}}{\le} Nq_2t^{-1}.
\end{align}
After plugging into  \eqref{luh} the expression
$$ 
W(t)=\sum\nolimits_{j=1}^{N}\Big(\frac{h }{t}-\frac{h }{t-\tau_j}\Big)(1-r_j(t))-\sum\nolimits_{j=1}^{N}\frac{\phi(t-\tau_j)}{t-\tau_j}(1-r_j(t)),
$$
we get
\begin{align*}
\frac{\phi(t)}{t}&= \rE_{ty}\Big(\sum\nolimits_{j=1}^{N}\frac{\phi(t-\tau_j)}{t-\tau_j}r_j(t)\Big)+\rE_{ty}\Big(\sum\nolimits_{j=1}^{N}\Big(\frac{h }{t-\tau_j}-\frac{h}{t}\Big)(1-r_j(t) )\Big)+o(t^{-2}),\quad t\to\infty. 
\end{align*}
The latter expectation is non-negative, and for an arbitrary $\epsilon>0$, it has the following upper bound
\begin{align*}
 \rE_{ty}\Big(\sum\nolimits_{j=1}^{N}\Big(\frac{h }{t-\tau_j}-\frac{h}{t}\Big)(1-r_j(t) )\Big)
\stackrel{ \eqref{stal}}{\le} q_2\epsilon\rE_{ty}\Big(\sum\nolimits_{j=1}^{N}\Big(\frac{h }{t-\tau_j}-\frac{h}{t} \Big)\Big)+\frac{q_2h}{(1-y)t^2}\rE(N^2;N>  t\epsilon).
\end{align*}
Thus, in view of Lemma \ref{L2},
\begin{align*}
\frac{\phi(t)}{t}&= \rE_{ty}\Big(\sum\nolimits_{j=1}^{N}\frac{\phi(t-\tau_j)}{t-\tau_j}r_j(t)\Big)+o(t^{-2}),\quad t\to\infty. 
\end{align*}
Multiplying this relation by $t$, we arrive at \eqref{afo}.

\subsection{Proof of $\phi(t)\to 0$}\label{end}

Recall \eqref{mt}. If  the non-decreasing function 
$$M(t):=\max_{1\le j\le t} m(j)$$ 
is bounded from above, then $\phi(t)=O(\frac{1}{\ln t})$ proving that $\phi(t)\to 0$ as $t\to\infty$. If $M(t)\to\infty$ as $t\to\infty$, then there is an integer-valued sequence $0<t_1<t_2<\ldots,$  such that the sequence $M_n:=M(t_n)$ is strictly increasing and converges to infinity.
In this case,
\begin{equation}\label{liv}
m(t)\le M_{n-1}<M_n,\quad 1\le t< t_n,\quad m(t_n)=M_n,\quad n\ge1.
\end{equation}
Since $|\phi(t)|\le \frac{M_{n}}{\ln t_{n}}$ for $t_n\le t<t_{n+1}$, to finish the proof of $\phi(t)\to 0$, it remains to verify that 
\begin{equation}\label{dog}
 M_{n}=o(\ln t_{n}),\quad n\to\infty.
\end{equation}

Fix an arbitrary $y\in(0,1)$. Putting $t=t_n$ in \eqref{elin} and using \eqref{liv}, we find
\begin{align*}
M_n\le M_n\rE_{t_ny}\Big(\sum\nolimits_{j=1}^{N}r_j(t_n)\frac{t_n\ln t_n}{(t_n-\tau_j)\ln(t_n-\tau_j)}\Big)+(t_n^{-1}\ln t_n)o_n.
\end{align*}
Here and elsewhere, $o_n$  stands for a non-negative sequence  such that $o_n\to0$ as $n\to\infty$. In different formulas, the sign $o_n$ represents different such sequences.
Since
$$
0\le  \frac{t\ln t}{(t-u)\ln (t-u)}-1\le  \frac{u(1+\ln t)}{(t-u)\ln (t-u)},\quad 0\le u< t-1,
$$
and $r_j(t_n)\in[0,1]$, it follows that
\begin{align*}
M_n-M_n\rE_{t_ny}\Big(\sum\nolimits_{j=1}^{N}r_j(t_n)\Big)&\le M_n\rE_{t_ny}\Big(\sum\nolimits_{j=1}^{N}\frac{\tau_j(1+\ln t_n)}{t_n(1-y)\ln (t_n(1-y))}\Big)+(t_n^{-1}\ln t_n)o_n.
\end{align*}
Recalling that $a=\rE(\sum_{j=1}^{N}\tau_j)$, observe that
\begin{align*}
\rE_{t_ny}\Big(\sum\nolimits_{j=1}^{N}\frac{\tau_j(1+\ln t_n)}{t_n(1-y)\ln (t_n(1-y))}\Big)\le \frac{a(1+\ln t_n)}{t_n(1-y)\ln (t_n(1-y))}
= (a(1-y)^{-1}+o_n)t_n^{-1}.
\end{align*}
Combining the last two relations, we conclude
\begin{align}\label{alt}
M_n\rE_{t_ny}\Big(\sum\nolimits_{j=1}^{N}(1-r_j(t_n))\Big)&\le a(1-y)^{-1}t_n^{-1}M_n +t_n^{-1}(M_n+\ln t_n)o_n.
\end{align}

Now it is time to unpack the term $r_j(t)$. By Lemma \ref{L4} with \eqref{sal},
%
$$
1-r_j(t)=\sum_{i=1}^{j-1}\frac{h +\phi(t-\tau_i)}{t-\tau_i}+(N-j)\frac{h }{t}-R_j(t),
$$
where provided $\tau_j\le ty$, 
$$
0\le R_j(t)\le Nq_2t^{-1}(1-y)^{-1},\quad R_j(t)\le N^2q_2^2t^{-2}(1-y)^{-2},\quad t>t^*,
$$
for a sufficiently large $t^*$. This allows us to rewrite \eqref{alt} in the form
\begin{align*}
M_n\rE_{t_ny}\Big(\sum\nolimits_{j=1}^{N}&\Big(\sum_{i=1}^{j-1}\frac{h +\phi(t_n-\tau_i)}{t_n-\tau_i}+(N-j)\frac{h }{t_n}\Big)\Big)\\
&\le M_n\rE_{t_ny}\Big(\sum\nolimits_{j=1}^{N} R_j(t_n)\Big)+a(1-y)^{-1}t_n^{-1}M_n +t_n^{-1}(M_n+\ln t_n)o_n.
\end{align*}
To estimate the last expectation, observe that if $\tau_j\le ty$, then for any $\epsilon>0$,
$$
R_j(t)\le Nq_2t^{-1}(1-y)^{-1} 1_{\{N>t\epsilon\}}+ N^2q_2^2t^{-2} (1-y)^{-2}1_{\{N\le t\epsilon\}},\quad t>t^*.
$$
implying that for sufficiently large $n$,
$$\rE_{t_ny}\Big(\sum\nolimits_{j=1}^{N}R_j(t_n)\Big)
\le q_2t_n^{-1}(1-y)^{-1}\rE(N^{2} ; N>  t_n\epsilon)+ q_2^2\epsilon t_n^{-1}(1-y)^{-2}\rE(N^2),$$
so that
\begin{align*}
M_n\rE_{t_ny}\Big(\sum\nolimits_{j=1}^{N}\Big(\sum\nolimits_{i=1}^{j-1}\frac{h +\phi(t_n-\tau_i)}{t_n-\tau_i}+(N-j)\frac{h }{t_n}\Big)\Big)
\le a(1-y)^{-1}t_n^{-1}M_n +t_n^{-1}(M_n+\ln t_n)o_n.
\end{align*}
Since
\[\sum\nolimits_{j=1}^{N}\sum\nolimits_{i=1}^{j-1}\Big(\frac{h}{t_n-\tau_i}- \frac{h }{t_n}\Big)\ge0,\]
we obtain
\begin{align*}
M_n\rE_{t_ny}\Big(\sum\nolimits_{j=1}^{N}\Big(\sum_{i=1}^{j-1}\frac{\phi(t_n-\tau_i)}{t_n-\tau_i}+(N-1)\frac{h }{t_n}\Big)\Big)
\le a(1-y)^{-1}t_n^{-1}M_n +t_n^{-1}(M_n+\ln t_n)o_n.
\end{align*}

By \eqref{cal} and \eqref{ca}, we have $\phi(t)\ge q_1-h$ for $t\ge t_0$. Thus for $\tau_j\le L\le  t_ny$ and sufficiently large $n$, 
$$\frac{\phi(t_n-\tau_i)}{t_n-\tau_i}\stackrel{}{\ge}  \frac{q_1-h}{t_n(1-y)}.$$
This gives
\[\sum\nolimits_{j=1}^{N}\Big(\sum_{i=1}^{j-1}\frac{\phi(t_n-\tau_i)}{t_n-\tau_i}+(N-1)\frac{h }{t_n}\Big)\ge \Big(h+\frac{q_1-h }{2(1-y)}\Big)t_n^{-1}N(N-1),\]
which after multiplying by $t_nM_n$ and taking expectations, yields
\begin{align*}
\Big(h+\frac{q_1-h }{2(1-y)}\Big)M_n\rE_{t_ny}(N(N-1))
\le a(1-y)^{-1}M_n +(M_n+\ln t_n)o_n.
\end{align*}
Finally, since
$$ \rE_{t_ny}(N(N-1))\to2b,\quad n\to\infty,$$
we derive that for any $0<\epsilon<y<1$, there is a finite $n_\epsilon$ such that for all $n>n_\epsilon$,
$$M_n\Big(2bh(1-y)+bq_1-bh-a-\epsilon\Big) \le \epsilon\ln t_n.$$

By \eqref{stop}, we have $bh\ge a$, and therefore, 
$$2bh(1-y)+bq_1-bh-a-\epsilon\ge bq_1-2bhy-y.$$ 
Thus, choosing $y=y_0$ such that $bq_1-2bhy_0-y_0=\frac{bq_1}{2}$, we see that
$$\limsup_{n\to\infty}\frac{M_n}{\ln t_n} \le \frac{2\epsilon}{bq_1},$$
which entails \eqref{dog} as $\epsilon\to0$, concluding the proof of $\phi(t)\to 0$.

\section{Proof of Theorem \ref{thL}}\label{Lp1}
We will use the following notational agreements for the $k$-dimensional probability generating function
\[\rE(z_1^{Z(t_1)}\cdots z_k^{Z(t_k)})=\sum_{i_1=0}^\infty\ldots\sum_{i_k=0}^\infty\rP(Z(t_1)=i_1,\ldots, Z(t_k)=i_k)z_1^{i_1}\cdots z_k^{i_k},\]
with $0< t_1\le \ldots\le t_k$ and $z_1,\ldots,z_k\in[0,1]$. We denote
\[P_k(\bar t,\bar z):=P_k(t_1,\ldots,t_{n};z_1,\ldots,z_{k}):=\rE(z_1^{Z(t_1)}\cdots z_k^{Z(t_k)}),\]
and write for $t\ge0$,
\[P_k(t+\bar t,\bar z):=P_k(t+t_1,\ldots,t+t_{k};z_1,\ldots,z_{k}).\]
Moreover, 
for $0< y_1<\ldots<y_k$, we write
\[P_k(t\bar y,\bar z):=P_k(ty_1,\ldots,ty_{k};z_1,\ldots,z_{k}),\]
and assuming $0< y_1<\ldots<y_k<1$, 
\[P_k^*(t,\bar y,\bar z):=\rE(z_1^{Z(ty_1)}\cdots z_{k}^{Z(ty_{k})};Z(t)=0)=P_{k+1}(ty_1,\ldots,ty_k,t;z_1,\ldots,z_k,0).
\]

These notational agreements will be similarly applied to the functions
\begin{equation}\label{Q*}
Q_k(\bar t,\bar z):=1-P_k(\bar t,\bar z),\quad Q_k^*(t,\bar y,\bar z):=1-P_k^*(t,\bar y,\bar z).
\end{equation}
Our special interest is in the function
\begin{equation}\label{krik}
 Q_k(t):=Q_k(t+\bar t,\bar z),\quad 0= t_1< \ldots< t_k, \quad z_1,\ldots,z_k\in[0,1),
\end{equation}
to be viewed as a counterpart of the function $Q(t)$ treated by Theorem 2. Recalling the compound parameters $h=\frac{a+\sqrt{a^2+4bd}}{2b}$ and $c=4bda^{-2}$, put
\begin{equation}\label{hk}
h_k:=h\frac{1+\sqrt{1+cg_k}}{1+\sqrt{1+c}},\quad g_k:= g_k(\bar y,\bar z):=\sum_{i=1}^{k}z_1\cdots z_{i-1}(1-z_{i})y_{i}^{-2}.
\end{equation}
The key step of the proof of Theorem 1 is to show that for any given $1=y_1<y_2<\ldots<y_k$,
\begin{equation}\label{dm}
tQ_k(t)\to h_k,\quad t_i:=t(y_i-1), \quad i=1,\ldots,k,\quad t\to\infty.
\end{equation}
This is done following the steps of our proof of $tQ(t)\to h$ given in Section \ref{out}. 

Unlike $Q(t)$, the function $Q_k(t)$ is not monotone over $t$. However, monotonicity of $Q(t)$ was used in the proof of Theorem 2 only in the proof of \eqref{ca}. The corresponding statement 
$$ 0<q_1\le tQ_k(t)\le q_2<\infty,\quad t\ge t_0,
$$
follows from the bounds $(1-z_1)Q(t)\le Q_k(t)\le Q(t)$, which hold due to monotonicity of the underlying generating functions over
$z_1,\ldots,z_{n}$. Indeed,
\[Q_k(t)\le Q_k(t, t+t_2,\ldots,t+t_{k};0,\ldots,0)= Q(t),\]
and on the other hand,
\[Q_k(t)= Q_k(t,t+t_2,\ldots,t+t_{k};z_1,\ldots,z_k)= \rE(1-z_1^{Z(t)}z_2^{Z(t+t_2)}\cdots z_k^{Z(t+t_k)})\ge \rE(1-z_1^{Z(t)}),\]
where 
\[ \rE(1-z_1^{Z(t)})\ge  \rE(1-z_1^{Z(t)};Z(t)\ge1)\ge (1-z_1)Q(t).\]

\subsection{Proof of \ $\boldsymbol{tQ_k(t)\to h_k}$}\label{Lup}

The branching property \eqref{CD} of the GWO-process  gives
\[ \prod_{i=1}^{k} z_i^{Z(t_i)}=\prod_{i=1}^{k} z_i^{1_{\{L>t_i\}}}\prod\nolimits_{j=1}^{N} z_i^{Z_j(t_i-\tau_j)}.\]
Given $0< t_1<\ldots<t_k< t_{k+1}=\infty$, we use
\begin{align*}
\prod_{i=1}^{k} z_i^{1_{\{L>t_i\}}}&=1_{\{L\le t_1\}}+\sum_{i=1}^{k}z_1\cdots z_{i}1_{\{t_{i}<L\le t_{i+1}\}},
\end{align*}
to deduce the following counterpart of  \eqref{ejp} 
\begin{align*}
P_k(\bar t,\bar z)&=\rE_{t_1}\Big(\prod_{j=1}^{N}P_k(\bar t-\tau_j,\bar z)\Big)+\sum_{i=1}^{k}z_1\cdots z_{i}\rE\Big(\prod_{j=1}^{N}P_k(\bar t-\tau_j,\bar z); 
t_{i}<L\le t_{i+1}\Big),
\end{align*}
which entails
\begin{align}\label{apes}
P_k(\bar t,\bar z)&=\rE_{t_1}\Big(\prod_{j=1}^{N}P_k(\bar t-\tau_j,\bar z)\Big)+\sum_{i=1}^{k}z_1\cdots z_{i} \rP(t_{i}<L\le t_{i+1}) \nonumber\\
&-\sum_{i=1}^{k}z_1\cdots z_{i} \rE\Big(1-\prod_{j=1}^{N}P_k(\bar t-\tau_j,\bar z); t_{i}<L\le t_{i+1}\Big).
\end{align}
Using this relation we establish the following counterpart of Lemma \ref{fQd}.

\begin{lemma}\label{fad}
Consider function \eqref{krik} and put 
$P_k(t):=1-Q_k(t)=P_k(t+\bar t,\bar z)$.
For $0<u<t$, the relation
 \begin{align} 
\Phi(h_k t^{-1})&= \rP(L> t)-\sum_{i=1}^{k}z_1\cdots z_{i}\rP(t+t_i<L\le t+t_{i+1}) \nonumber \\
&+\rE_u\Big(\sum\nolimits_{j=1}^{N}Q_k(t-\tau_j)\Big)-Q_k(t)+\rE_u(W_k(t))+D_k(u,t), \label{arr}
\end{align}
holds with $t_{k+1}=\infty$,
 \begin{align} \label{tWt}
W_k(t):=(1-h_k t^{-1})^{N}+Nh_k t^{-1}-\sum\nolimits_{j=1}^{N}Q_k(t-\tau_j)-\prod\nolimits_{j=1}^{N}P_k(t-\tau_j)
\end{align}
and
 \begin{align} \label{tDut}
D_k(u,t):=\ &\rE\Big(1-\prod\nolimits_{j=1}^{N}P_k(t-\tau_j);u<L\le t\Big)+\rE\Big((1-h_k t^{-1})^{N} -1+Nh_k t^{-1};L> u\Big)  \nonumber\\
&+\sum_{i=1}^{k}z_1\cdots z_{i} \rE\Big(1-\prod_{j=1}^{N}P_k(t-\tau_j); t+t_{i}<L\le t+t_{i+1}\Big).
\end{align}

\end{lemma}

\begin{proof} 
According to \eqref{apes}, 
\begin{align*}
P_k(t)&=\rE_u\Big(\prod_{j=1}^{N}P_k(t-\tau_j)\Big)+\rE\Big(\prod\nolimits_{j=1}^{N}P_k(t-\tau_j);u<L\le  t\Big)
\\
&+\sum_{i=1}^{k}z_1\cdots z_{i} \rP(t+t_{i}<L\le t+t_{i+1})-\sum_{i=1}^{k}z_1\cdots z_{i} \rE\Big(1-\prod_{j=1}^{N}P_k(t-\tau_j); t+t_{i}<L\le t+t_{i+1}\Big).
\end{align*}
By the definition of $\Phi(\cdot)$,
\begin{align*} 
\Phi(h_k t^{-1})+1
&=\rE_u\Big((1-h_k t^{-1})^{N}+Nh_k t^{-1}\Big)+\rP(L> t)\\
&+\rE\Big((1-h_k t^{-1})^{N} -1+Nh_k t^{-1};L> u\Big)+\rP(u<L\le t),
\end{align*}
and after subtracting the two last equations, we get
\begin{align*}
\Phi(h_k t^{-1})+Q_k(t)&=\rE_u\Big((1-h_k t^{-1})^{N} +Nh_k t^{-1}-\prod\nolimits_{j=1}^{N}P_k(t-\tau_j)\Big)+\rP(L> t)\\
&-\sum_{i=1}^{k}z_1\cdots z_{i} \rP(t+t_{i}<L\le t+t_{i+1})+D_k(u,t)
\end{align*}
with $D_k(u,t)$ satisfying \eqref{tDut}.
After a rearrangement, relation \eqref{arr} follows together with  \eqref{tWt}.
\end{proof}

With Lemma \ref{fad} in hand, convergence \eqref{dm} is proven applying almost exactly the same argument used in the proof of $tQ(t)\to h$. An important new feature emerges due to the additional term in the asymptotic relation defining the limit $h_k$. Let $1=y_1<y_2<\ldots<y_k<y_{k+1}=\infty$. Since 
\begin{align*}
\sum\nolimits_{i=1}^{k}z_1\cdots z_{i}\rP(ty_{i}<L\le ty_{i+1})\sim d t^{-2}\sum_{i=1}^{k}z_1\cdots z_{i}(y_{i}^{-2}-y_{i+1}^{-2}),
\end{align*}
we see that
\begin{align*}
\rP(L> t)-\sum\nolimits_{i=1}^{k}z_1\cdots z_{i}\rP(ty_{i}<L\le ty_{i+1})\sim dg_k t^{-2},
\end{align*}
where $g_k$ is defined by \eqref{hk}.
Assuming $0\le z_1,\ldots,z_k<1$, we ensure that $g_k>0$, and as a result, we arrive at a counterpart of the quadratic equation \eqref{stop}, 
\[
bh_k^2=ah_k+dg_k,
\]
which gives 
\[
h_k=\frac{a+\sqrt{a^2+4bdg_k}}{2b}=h\frac{1+\sqrt{1+cg_k}}{1+\sqrt{1+c}},\]
justifying our definition \eqref{hk}.
We conclude that for $k\ge1$,
\begin{equation}\label{love}
\frac{Q_k(t\bar y,\bar z)}{Q(t)}\to \frac{1+\sqrt{1+c\sum\nolimits_{i=1}^{k}z_1\cdots z_{i-1}(1-z_{i})y_{i}^{-2}}}{1+\sqrt{1+c}},\quad 1=y_1<\ldots< y_k,\quad 
0\le z_1,\ldots,z_k<1.
\end{equation}

\subsection{Conditioned generating functions}\label{Send}
To finish the proof of Theorem 1, consider the generating functions conditioned on the survival of the GWO-process. 
Given \eqref{mansur} with $j\ge1$, we have
\begin{align*}
Q(t)\rE&(z_1^{Z(ty_1)}\cdots z_k^{Z(ty_k)}|Z(t)>0)=\rE(z_1^{Z(ty_1)}\cdots z_k^{Z(ty_k)};Z(t)>0)\\
&=P_k(t\bar y,\bar z)-\rE(z_1^{Z(ty_1)}\cdots z_k^{Z(ty_k)};Z(t)=0)\stackrel{\eqref{Q*}}{=}Q_j^*(t,\bar y,\bar z)-Q_k(t\bar y,\bar z),
\end{align*}
and therefore,
\[\rE(z_1^{Z(ty_1)}\cdots z_k^{Z(ty_k)}|Z(t)>0)=\frac{Q_j^*(t,\bar y,\bar z)}{Q(t)}-\frac{Q_k(t\bar y,\bar z)}{Q(t)}.\]
Similarly, if \eqref{mansur} holds with $j=0$, then
\[\rE(z_1^{Z(ty_1)}\cdots z_k^{Z(ty_k)}|Z(t)>0)=1-\frac{Q_k(t\bar y,\bar z)}{Q(t)}.\]

Letting $t'=ty_1$, we get
\[\frac{Q_k(t\bar y,\bar z)}{Q(t)}=\frac{Q_k(t',t'y_2/y_1,\ldots,t'y_k/y_1)}{Q(t')}\frac{Q(ty_1)}{Q(t)},\]
and applying relation \eqref{love},
\begin{equation*}
\frac{Q_k(t\bar y,\bar z)}{Q(t)}\to \frac{1+\sqrt{1+\sum\nolimits_{i=1}^{k}z_1\cdots z_{i-1}(1-z_{i})\Gamma_i}}{(1+\sqrt{1+c})y_1},
\end{equation*}
where $\Gamma_i=c({y_1}/{y_i} )^2$.
On the other hand, since
\[Q_j^*(t,\bar y,\bar z)=Q_{j+1}(ty_1,\ldots,ty_j,t;z_1,\ldots,z_j,0), \quad j\ge1,\]
we also get
\begin{equation*}
\frac{Q_j^*(t,\bar y,\bar z)}{Q(t)}\to \frac{1+\sqrt{1+\sum\nolimits_{i=1}^{j}z_1\cdots z_{i-1}(1-z_{i})\Gamma_i+cz_1\cdots z_{j}y_1^2}}{(1+\sqrt{1+c})y_1}.
\end{equation*}
We conclude that as stated in Section \ref{main},
\begin{align*}
\rE(z_1^{Z(ty_1)}\cdots z_k^{Z(ty_k)}|Z(t)>0)\to \rE(z_1^{\eta(y_1)}\cdots z_k^{\eta(y_k)}).
\end{align*}

\subsection*{Acknowledgements}
The author is grateful to two anonymous referees for their valuable comments, corrections, and suggestions that helped to enhance the readability of the paper.

\end{document}